\documentclass{amsart}
\usepackage[bookmarks=true,colorlinks,urlcolor=black,citecolor=black,linkcolor=black]{hyperref}
\usepackage{mathrsfs}
\usepackage{amsmath}
\usepackage{mathtools}
\usepackage{amssymb}
\usepackage{graphicx}
\usepackage{latexsym}
\newtheorem{theorem}{Theorem}

\theoremstyle{definition}

\theoremstyle{definition}

\newtheorem*{remark}{Remark}
\graphicspath{{./Figures/}{./30baiIMO2016/}{C:/Users/hung/Dropbox/Hung/Figures/}}
\begin{document}
	\title[A family of weighted Erd\H{o}s-Mordell inequality and applications]{A family of weighted Erd\H{o}s-Mordell inequality and applications}
	
	\author[Tran Quang Hung]{Tran Quang Hung}
	\address{High School for Gifted Students, Hanoi, Vietnam}
	\email{tranquanghung@hus.edu.vn}
	
	\keywords{Geometric inequality, Erd\H{o}s-Mordell inequality, weighted inequality}
	\subjclass[2010]{51M04, 51M16}
	\maketitle

\begin{abstract}We establish some new generalizations of Erd\H{o}s-Mordell inequality by adding weights to its terms. Using these generalizations, we derived strengthened versions of the original Erd\H{o}s-Mordell inequality. We also found two other variations of Erd\H{o}s-Mordell inequality to be generalizable in the same way.
\end{abstract}

\section{Introduction} Throughout this article we denote the distance between two points $X$ and $Y$ simply by $XY$. Let $P$ be an interior point of the triangle $ABC$. Then
\begin{equation*}PA+PB+PC\ge 2(d_a+d_b+d_c),\end{equation*}
where $d_a$, $d_b$, and $d_c$ are the distances from $P$ to the sides $BC$, $CA$, $AB$, respectively, and equality holds if and only if triangle $ABC$ is equilateral and $P$ is its center.

\begin{figure}[htbp]
	\begin{center}\scalebox{0.7}{\includegraphics{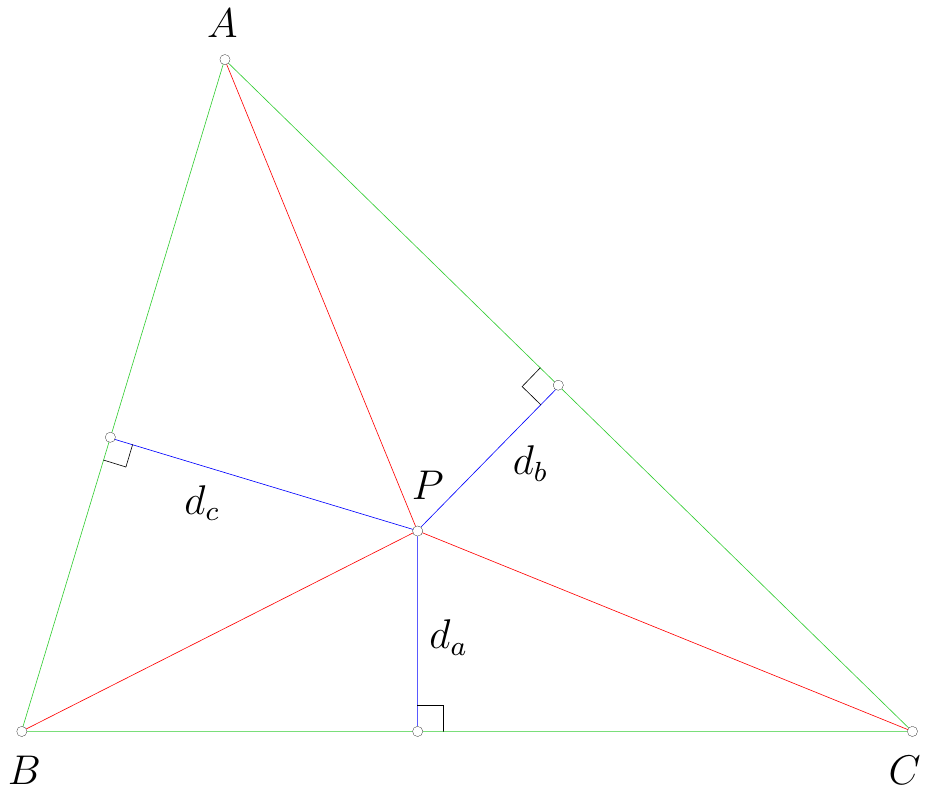}}\end{center}
	\caption{Erd\H{o}s-Mordell inequality.}
	\label{fig1}
\end{figure}

This is the well-known and elegant Erd\H{o}s–Mordell inequality. It was named after Erd\H{o}s, who conjectured it in 1935 \cite{1}, and Mordell, who was the first to prove it in 1937 \cite{2}. Throughout the history of elementary geometry, there has been great deals of proofs, developments and generalizations to this classical result, see \cite{2,3,11,12,13,14,15,16,18,19,20a,20}.

In this article, following the ideas of generalizing Erd\H{o}s–Mordell inequality using weights from Dar and Gueron \cite{3} and Liu \cite{20a}, we developed a new generalization using a different type of weight. As a result, we discovered new weighted generalizations together with strengthened versions of the original Erd\H{o}s–Mordell inequality. In addition, the new weighting method is able to generalize two other variations of the original Erd\H{o}s–Mordell inequality. The first one is Barrow's inequality \cite{2}:
\begin{equation*}PA+PB+PC\ge 2(l_a+l_b+l_c),\end{equation*}
where $l_a$, $l_b$, and $l_c$ are the lengths of the bisectors of $\angle BPC$, $\angle CPA$, and $\angle APB$, respectively, and equality holds if and only if triangle $ABC$ is equilateral and $P$ is its center. The second variation is from Dao, Nguyen, and Pham \cite{15}:
\begin{equation*}R_A+R_B+R_C\ge 2(d_a+d_b+d_c),\end{equation*}
where $R_A$, $R_B$, and $R_C$ are the distances of $P$ from the tangents to circumcircle of triangle $ABC$ at $A$, $B$, and $C$, respectively, and equality holds if and only if triangle $ABC$ is equilateral and $P$ is its center.

We present the new weighted generalizations with proofs in Section 2 and the resulting some applications and strengthened inequalities will be given in Section 3.

\section{Main theorems and proofs} We begin with a new weighted version of Erd\H{o}s-Mordell inequality:

\begin{theorem}[A new weighted version of Erd\H{o}s-Mordell inequality]\label{thm0}Let $x$, $y$, $z$, $u$, $v$, and $w$ be positive real numbers such that $xyz=uvw=1$. Let $P$ be an interior point of the triangle $ABC$. Denote by $d_a$, $d_b$, and $d_c$ the distances of $P$ from the sides $AB$, $BC$, and $CA$, respectively. Then
	\begin{equation}\label{eq0}x(PA+u^3d_a)+y(PB+v^3d_b)+z(PC+w^3d_c)\ge 3(ud_a+vd_b+wd_c).
	\end{equation}
	Equality holds if and only if $x=y=z=u=v=w=1$, triangle $ABC$ is equilateral and $P$ is its center.
\end{theorem}
\begin{proof}With standard notations $BC=a$, $CA=b$, and $AB=c$, we recall a result in \cite{3}:
	\begin{equation}\label{eq3}PA\ge \frac{b\cdot d_c+c\cdot d_b}{a}.\end{equation}
	Equality in \eqref{eq3} holds if and only if $P$ lies on the line connecting $A$ to the circumcenter of $ABC$.
	Since $x$ is a positive real number,
	\begin{equation}\label{eq4}x(PA+u^3d_a)\ge x\left(\frac{b\cdot d_c+c\cdot d_b}{a}+u^3d_a\right).\end{equation}
	Similarly, with positive real numbers $y$ and $z$, we get
	\begin{equation}\label{eq5}y(PB+v^3d_b)\ge y\left(\frac{c\cdot d_a+a\cdot d_c}{b}+v^3d_b\right)\end{equation}
	and
	\begin{equation}\label{eq6}z(PC+w^3d_c)\ge z\left(\frac{a\cdot d_b+b\cdot d_a}{c}+w^3d_c\right)\end{equation}
	with analogous conditions for equality. Now summing the inequalities \eqref{eq4}, \eqref{eq5}, and \eqref{eq6}, we obtain
	\begin{multline}\label{eq7}x(PA+u^3d_a)+y(PB+v^3d_b)+z(PC+w^3d_c)\ge \left(y\cdot\frac{c}{b}+z\cdot\frac{b}{c}+xu^3\right)d_a+\\
	+\left(z\cdot\frac{b}{a}+x\cdot\frac{c}{a}+yv^3\right)d_b
	+\left(x\cdot\frac{a}{b}+y\cdot\frac{a}{b}+zw^3\right)d_c.
	\end{multline}
	and equality holds if and only if $P$ is circumcenter of $\triangle ABC$.
	
	Using the arithmetic mean-geometric mean inequality for three positive numbers, with notice that $xyz=1$, we have
	\begin{equation}\label{eq8}y\cdot\frac{c}{b}+z\cdot\frac{b}{c}+xu^3\ge 3\sqrt[3]{y\cdot\frac{c}{b}\cdot z\cdot\frac{b}{c}\cdot xu^3}=3u.
	\end{equation}
	with equality holds if and only if
	\begin{equation}\label{eq9}y\cdot\frac{c}{b}=z\cdot\frac{b}{c}=xu^3,
	\end{equation}
	which is equivalent to
	\begin{equation}\label{eq10}u^3\cdot\frac{x}{y}=\frac{c}{b}\quad\text{and}\quad u^3\cdot\frac{x}{z}=\frac{b}{c}.
	\end{equation}
	Similarly,
	\begin{equation}\label{eq11}z\cdot\frac{b}{a}+x\cdot\frac{c}{a}+yv^3\ge 3v,
	\end{equation}
	and
	\begin{equation}\label{eq12}x\cdot\frac{a}{b}+y\cdot\frac{a}{b}+zw^3\ge 3w
	\end{equation}
	Since $uvw=1$, the equalities hold if and only if
	\begin{equation}\label{eq14}x=y=z=u=v=w=1\quad\text{and}\quad a=b=c.
	\end{equation}
	Finally, combining \eqref{eq7}, \eqref{eq8}, \eqref{eq11}, and \eqref{eq12}, we get
	\begin{equation}\label{eq15}x(PA+u^3d_a)+y(PB+v^3d_b)+z(PC+w^3d_c)\ge 3(ud_a+vd_b+wd_c)
	\end{equation}
	with equality holds if and only if
	\begin{equation}\label{eq16}x=y=z=u=v=w=1\quad\text{and}\quad a=b=c.
	\end{equation}
	This completes the proof of Theorem \ref{thm0}.
\end{proof}
\begin{remark}In Theorem \ref{thm0}, where $u=v=w=1$, we obtain
	\begin{equation}x(PA+d_a)+y(PB+d_b)+z(PC+d_c)\ge 3(d_a+d_b+d_c).
	\end{equation}
	In Theorem \ref{thm0}, where $x=y=z=1$, we obtain
	\begin{equation}PA+PB+PC\ge (3u-u^3)d_a+(3v-v^3)d_b+(3w-w^3)d_c.
	\end{equation}
	Theses inequalities can be considered as other weighted versions of Erd\H{o}s-Mordell inequality.
\end{remark}

Similarly, we get a new weighted Dao-Nguyen-Pham's inequality in \cite{15}.

\begin{theorem}[A new weighted version of Dao-Nguyen-Pham's inequality]\label{thma0}Let $x$, $y$, $z$, $u$, $v$, and $w$ be positive real numbers such that $xyz=uvw=1$. Let $P$ be an interior point of the triangle $ABC$ with circumcircle $(\omega)$. Denote by $R_A$, $R_B$, and $R_C$ the distances of $P$ from the tangents to $(\omega)$ at $A$, $B$, $C$ respectively, $d_a$, $d_b$, and $d_c$ the distances of $P$ from the sides $AB$, $BC$, and $CA$, respectively. Then
	\begin{equation}\label{eqa0}x(R_A+u^3d_a)+y(R_B+v^3d_b)+z(R_C+w^3d_c)\ge 3(u d_a+v d_b+w d_c).
	\end{equation}
	Equality holds if and only if $x=y=z=u=v=w=1$, triangle $ABC$ is equilateral and $P$ is its center.
\end{theorem}
\begin{proof}With standard notations $BC=a$, $CA=b$, and $AB=c$, we recall a result of the second proof in \cite{15}:
	\begin{equation}\label{eq3a}R_A=\frac{b\cdot d_c+c\cdot d_b}{a}.\end{equation}
	Since $x$ is a positive real number, 
	\begin{equation}\label{eq4a}x(R_A+u^3d_a)= x\left(\frac{b\cdot d_c+c\cdot d_b}{a}+u^3d_a\right).\end{equation}
	Similarly, with positive real numbers $y$ and $z$, we get
	\begin{equation}\label{eq5a}y(R_B+v^3d_b)= y\left(\frac{c\cdot d_a+a\cdot d_c}{b}+v^3d_b\right)\end{equation}
	and
	\begin{equation}\label{eq6a}z(R_C+w^3d_c)= z\left(\frac{a\cdot d_b+b\cdot d_a}{c}+w^3d_c\right)\end{equation}
	Now summing the inequalities \eqref{eq4a}, \eqref{eq5a}, and \eqref{eq6a}, we obtain the identity
	\begin{multline}\label{eq7a}x(R_A+u^3d_a)+y(R_B+v^3d_b)+z(R_C+w^3d_c)= \left(y\cdot\frac{c}{b}+z\cdot\frac{b}{c}+xu^3\right)d_a+\\
	+\left(z\cdot\frac{b}{a}+x\cdot\frac{c}{a}+yv^3\right)d_b
	+\left(x\cdot\frac{a}{b}+y\cdot\frac{a}{b}+zw^3\right)d_c.
	\end{multline}
	Using the arithmetic mean-geometric mean inequality for three positive numbers, with notice that $xyz=1$, we have
	\begin{equation}\label{eq8a}y\cdot\frac{c}{b}+z\cdot\frac{b}{c}+xu^3\ge 3\sqrt[3]{y\cdot\frac{c}{b}\cdot z\cdot\frac{b}{c}\cdot xu^3}=3u.
	\end{equation}
	with equality holds if and only if
	\begin{equation}\label{eq9a}y\cdot\frac{c}{b}=z\cdot\frac{b}{c}=xu^3,
	\end{equation}
	which is equivalent to
	\begin{equation}\label{eq10a}u^3\cdot\frac{x}{y}=\frac{c}{b}\quad\text{and}\quad u^3\cdot\frac{x}{z}=\frac{b}{c}.
	\end{equation}
	Similarly,
	\begin{equation}\label{eq11a}z\cdot\frac{b}{a}+x\cdot\frac{c}{a}+yv^3\ge 3v,
	\end{equation}
	and
	\begin{equation}\label{eq12a}x\cdot\frac{a}{b}+y\cdot\frac{a}{b}+zw^3\ge 3w
	\end{equation}
	Since $uvw=1$, the equalities hold if and only if
	\begin{equation}\label{eq14a}x=y=z=u=v=w=1\quad\text{and}\quad a=b=c.
	\end{equation}
	Finally, combining \eqref{eq7a}, \eqref{eq8a}, \eqref{eq11a}, and \eqref{eq12a}, we get
	\begin{equation}\label{eq15a}x(R_A+u^3d_a)+y(R_B+v^3d_b)+z(R_C+w^3d_c)\ge 3(u d_a+v d_b+wd_c)
	\end{equation}
	with equality holds if and only if
	\begin{equation}\label{eq16a}x=y=z=u=v=w=1\quad\text{and}\quad a=b=c.
	\end{equation}
	This completes the proof of Theorem \ref{thma0}.
\end{proof}
\begin{remark}Since $R_A\le PA$, $R_b\le PB$, and $R_c\le PC$, Theorem \ref{thma0} can be considered as a strengthened version of Theorem \ref{thm0}. In Theorem \ref{thma0}, let $u=v=w=1$ or $x=y=z=1$, we obtain two weighted inequalities
	\begin{equation}x(R_A+d_a)+y(R_B+d_b)+z(R_C+d_c)\ge 3(d_a+d_b+d_c),
	\end{equation}
	and
	\begin{equation}R_A+R_B+R_C\ge (3u-u^3)d_a+(3v-v^3)d_b+(3w-w^3)d_c.
	\end{equation}
	They are considered as other weighted versions of Dao-Nguyen-Pham's inequality.
\end{remark}
By the same way, we have a generalization for Barrow's inequality. Similar to Barow's proof, we will use Wolstenholme's inequality \cite{4} to prove this generalization.
\begin{theorem}[A new weighted version of Barrow's inequality]\label{thm1}Let $x$, $y$, $z$, $u$, $v$, and $w$ be positive real numbers such that $xyz=uvw=1$. Let $P$ be an interior point of the triangle $ABC$. Denote by $l_a$, $l_b$, and $l_c$ the lengths of the bisectors of $\angle BPC$, $\angle CPA$, and $\angle APB$, respectively. Then
	\begin{equation}\label{eq1}x(PA+u^3l_a)+y(PB+v^3l_b)+z(PC+w^3l_c)\ge (3ul_a+vl_b+wl_c).
	\end{equation}
	Equality holds if and only if $x=y=z=u=v=w=1$, triangle $ABC$ is equilateral and $P$ is its center.
\end{theorem}
\begin{proof}Let $\angle BPC=\alpha$, $\angle CPA=\beta$, angle $APB=\gamma$, using formula of bisector, we have
	\begin{equation}\label{eq19}l_a=\frac{PB\cdot PC\cos\frac{\alpha}{2}}{PB+PC},\end{equation}
	this is equivalent to
	\begin{equation}\label{eq20}2\sqrt{PB\cdot PC}\cos\frac{\alpha}{2}=\left(\sqrt{\frac{PB}{PC}}+\sqrt{\frac{PC}{PB}}\right)l_a.\end{equation}
	Multiplying both sides of \eqref{eq20} with $\sqrt{yz}$, we have
	\begin{equation}\label{eq21}2\sqrt{yPB\cdot zPC}\cos\frac{\alpha}{2}=\sqrt{yz}\left(\sqrt{\frac{PB}{PC}}+\sqrt{\frac{PC}{PB}}\right)l_a,\end{equation}
	and similarly,
	\begin{equation}\label{eq22}2\sqrt{zPC\cdot xPA}\cos\frac{\beta}{2}=\sqrt{zx}\left(\sqrt{\frac{PC}{PA}}+\sqrt{\frac{PA}{PC}}\right)l_b,\end{equation}
	\begin{equation}\label{eq23}2\sqrt{xPA\cdot yPB}\cos\frac{\gamma}{2}=\sqrt{xy}\left(\sqrt{\frac{PA}{PB}}+\sqrt{\frac{PB}{PA}}\right)l_c.\end{equation}
	Since $\frac{\alpha}{2}+\frac{\beta}{2}+\frac{\gamma}{2}=\pi$, using Wolstenholme's inequality \cite{4}, we get
	\begin{multline}\label{eq24}2\sqrt{PB\cdot PC}\cos\frac{\alpha}{2}+2\sqrt{zPC\cdot xPA}\cos\frac{\beta}{2}+\\+2\sqrt{xPA\cdot yPB}\cos\frac{\gamma}{2}\le xPA+yPB+zPC.\end{multline}
	From \eqref{eq21}, \eqref{eq22}, \eqref{eq23}, and \eqref{eq24}, we obtain
	\begin{multline}\label{eq25}xPA+yPB+zPC\ge \sqrt{yz}\left(\sqrt{\frac{PB}{PC}}+\sqrt{\frac{PC}{PB}}\right)l_a+\\+\sqrt{zx}\left(\sqrt{\frac{PC}{PA}}+\sqrt{\frac{PA}{PC}}\right)l_b+\sqrt{xy}\left(\sqrt{\frac{PA}{PB}}+\sqrt{\frac{PB}{PA}}\right)l_c.
	\end{multline}
	Adding two sides of \eqref{eq25} with $xu^3l_a+yv^3l_b+zw^3l_c$, we get
	\begin{multline}\label{eq26}x(PA+u^3l_a)+y(PB+v^3l_b)+z(PC+w^3l_c)\ge\\ \left(\sqrt{yz}\sqrt{\frac{PB}{PC}}+\sqrt{yz}\sqrt{\frac{PC}{PB}}+xu^3\right)l_a+\left(\sqrt{zx}\sqrt{\frac{PC}{PA}}+\sqrt{zx}\sqrt{\frac{PA}{PC}}+yv^3\right)l_b+\\ +\left(\sqrt{xy}\sqrt{\frac{PA}{PB}}+\sqrt{xy}\sqrt{\frac{PB}{PA}}+zw^3\right)l_c.
	\end{multline}
	Since $xyz=1$, using arithmetic mean-geometric mean inequality for three positive numbers, we have
	\begin{equation}\label{eq27}\sqrt{yz}\sqrt{\frac{PB}{PC}}+\sqrt{yz}\sqrt{\frac{PC}{PB}}+xu^3\ge 3\sqrt[3]{\sqrt{yz}\sqrt{\frac{PB}{PC}}\cdot\sqrt{yz}\sqrt{\frac{PC}{PB}}\cdot xu^3}=3u,\end{equation}
	and similarly,
	\begin{equation}\label{eq28}\sqrt{zx}\sqrt{\frac{PC}{PA}}+\sqrt{zx}\sqrt{\frac{PA}{PC}}+y\ge  3v,\end{equation}
	\begin{equation}\label{eq29}\sqrt{xy}\sqrt{\frac{PA}{PB}}+\sqrt{xy}\sqrt{\frac{PB}{PA}}+z\ge  3w.\end{equation}
	Combining \eqref{eq26}, \eqref{eq27}, \eqref{eq28}, and \eqref{eq29}, we have
	\begin{equation}\label{eq30}x(PA+u^3l_a)+y(PB+v^3l_b)+z(PC+w^3l_c)\ge 3(ul_a+vl_b+wl_c),\end{equation}
	with equality holds if and only if
	\begin{equation}\label{eq30a}x=y=z=u=v=w=1,\end{equation}
	triangle $ABC$ is equilateral and $P$ is its center.
	This completes the proof of Theorem \ref{thm1}.
\end{proof}
\begin{remark}Similarly, we have two particular cases as new weighted versions of Barrow's inequality. In Theorem \ref{thm1}, let $u=v=w=1$ then
	\begin{equation}x(PA+l_a)+y(PB+l_b)+z(PC+l_c)\ge 3(l_a+l_b+l_c),
	\end{equation}
	let $x=y=z=1$ then
	\begin{equation}PA+PB+PC\ge (3u-u^3)l_a+(3v-v^3)l_b+(3w-w^3)l_c.
	\end{equation}
\end{remark}

\section{Some applications} In this section, we apply Theorem \ref{thm0}, Theorem \ref{thma0}, and Theorem \ref{thm1} to get some strengthened versions of Erd\H{o}s-Mordell inequalities and its variations.

\begin{theorem}\label{thm2}Let $u$, $v$, and $w$ be positive real numbers such that $uvw=1$. Let $P$ be an interior point of the triangle $ABC$. Denote by $d_a$, $d_b$, and $d_c$ the distances of $P$ from sides $AB$, $BC$, and $CA$, respectively. Then
	\begin{equation}\label{eq31}(PA+u^3d_a)(PB+v^3d_b)(PC+w^3d_c)\ge (ud_a+vd_b+wd_c)^3.\end{equation}
	Equality holds if and only if $u=v=w=1$, triangle $ABC$ is equilateral and $P$ is its center.
\end{theorem}
\begin{proof}From Theorem \ref{thm0}, let \begin{equation}\label{eq31aa}x=\frac{\sqrt[3]{(PA+u^3d_a)(PB+v^3d_b)(PC+2^3d_c)}}{PA+u^3d_a},\end{equation} 
	\begin{equation}\label{eq31bb}y=\frac{\sqrt[3]{(PA+u^3d_a)(PB+v^3d_b)(PC+w^3d_c)}}{PB+v^3d_b},\end{equation}
	and
	\begin{equation}\label{eq31cc}z=\frac{\sqrt[3]{(PA+u^3d_a)(PB+v^3d_b)(PC+w^3d_c)}}{PC+w^3d_c}.\end{equation} 
	It is easily seen that $x$, $y$, and $z$ are positive numbers satisfied the condition $xyz=1$, we get
	\begin{equation}\label{eq31b}3\sqrt[3]{(PA+u^3d_a)(PB+v^3d_b)(PC+w^3d_c)}\ge 3 (ud_a+vd_b+wd_c)\end{equation}
	or
	\begin{equation}\label{eq31c}(PA+u^3d_a)(PB+v^3d_b)(PC+w^3d_c)\ge (ud_a+vd_b+wd_c)^3.\end{equation}
	Equality holds if and only if $u=v=w=1$, triangle $ABC$ is equilateral and $P$ is its center. This completes the proof of Theorem \ref{thm2}.
\end{proof}
\begin{remark}From Theorem \ref{thm2}, let $u=v=w=1$, we get a strengthened version of Erd\H{o}s-Mordell inequality
	\begin{equation}\label{eq31d}(PA+d_a)(PB+d_b)(PC+d_c)\ge (d_a+d_b+d_c)^3\end{equation}
	because
	\begin{multline}(d_a+d_b+d_c)^3\le (PA+d_a)(PB+d_b)(PC+d_c)\le\\ \le\left(\frac{PA+d_a+PB+d_b+PC+d_c}{3}\right)^3.\end{multline}
	which is equivalent to
	\begin{equation}3(d_a+d_b+d_c)\le PA+d_a+PB+d_b+PC+d_c\end{equation}
	or
	\begin{equation}2(d_a+d_b+d_c)\le PA+PB+PC.\end{equation}
	Notice that the inequality \eqref{eq31d} has appeared in \cite{19} but with a different proof. Theorem \ref{thm2} is a generalization of inequality \eqref{eq31d}.
\end{remark}

Likewise, we have the similar application of Theorem \ref{thma0} and obtain a strengthened version of Dao-Nguyen-Pham's inequality.

\begin{theorem}\label{thm8}Let $u$, $v$, and $w$ be positive real numbers such that $uvw=1$. Let $P$ be an interior point of the triangle $ABC$ with circumcircle $(\omega)$. Denote by $R_A$, $R_B$, and $R_C$ the distances of $P$ from the tangents to $(\omega)$ at $A$, $B$, $C$ respectively, $d_a$, $d_b$, and $d_c$ the distances of $P$ from the sides $AB$, $BC$, and $CA$, respectively. Then
	\begin{equation}(R_A+u^3d_a)(R_B+v^3d_b)(R_C+w^3d_c)\ge (ud_a+vd_b+wd_c)^3.
	\end{equation}
	Equality holds if and only if $u=v=w=1$, $P$ coincides one of the vertices of the triangle $ABC$.
\end{theorem}
From Theorem \ref{thm8}, let $u=v=w=1$, we get a strengthened version of Dao-Nguyen-Pham's inequality also it is strengthened version of inequality \eqref{eq31d}
\begin{equation}\label{eq31e}(R_A+d_a)(R_B+d_b)(R_C+d_c)\ge (d_a+d_b+d_c)^3.\end{equation}
With the same idea, we get the following theorem as a particular case of Theorem \ref{thm1}
\begin{theorem}\label{thm3}Let $u$, $v$, and $w$ be positive real numbers such that $uvw=1$. Let $P$ be an interior point of the triangle $ABC$. Denote by $l_a$, $l_b$, and $l_c$ the lengths of the bisectors of $\angle BPC$, $\angle CPA$, and $\angle APB$, respectively. Then
	\begin{equation}\label{eq32}(PA+u^3l_a)(PB+v^3l_b)(PC+w^3l_c)\ge (ul_a+vl_b+wl_c)^3.\end{equation}
	Equality holds if and only if $u=v=w=1$, triangle $ABC$ is equilateral and $P$ is its center.
\end{theorem}
From Theorem \ref{thm3}, let $u=v=w=1$, we obtain a strengthened version of Barrow's inequality
\begin{equation}\label{eq31f}(PA+l_a)(PB+l_b)(PC+l_c)\ge (l_a+l_b+l_c)^3.\end{equation}

Notice that we also have a strengthened version of the weighted Barrow's inequality from \eqref{eq25} by rewriting it as follows
\begin{theorem}[Another strengthened version of weighted Barrow's inequality]\label{thm4}Let $x$, $y$, and $z$ be any three positive real numbers. Let $P$ be an interior point of the triangle $ABC$. Denote by $l_a$, $l_b$, and $l_c$ be the lengths of the bisectors of $\angle BPC$, $\angle CPA$, and $\angle APB$, respectively. Then
	\begin{multline}\label{eq33}xPA+yPB+zPC\ge \sqrt{yz}\left(\sqrt{\frac{PB}{PC}}+\sqrt{\frac{PC}{PB}}\right)l_a+\\+\sqrt{zx}\left(\sqrt{\frac{PC}{PA}}+\sqrt{\frac{PA}{PC}}\right)l_b+\sqrt{xy}\left(\sqrt{\frac{PA}{PB}}+\sqrt{\frac{PB}{PA}}\right)l_c.
	\end{multline}
	Equality holds if and only if $x=y=z$, triangle $ABC$ is equilateral and $P$ is its center.
\end{theorem}
We notice that, a form without weights of Theorem \ref{thm4} has appeared in \cite{16a}. Using Theorem \ref{thm4}, we get another strengthened versions of Barrow's inequality as follows
\begin{theorem}[Another strengthened version of Barrow's inequality]\label{thm5}Let $P$ be an interior point of the triangle $ABC$. Denote by $l_a$, $l_b$, and $l_c$ the lengths of the bisectors of $\angle BPC$, $\angle CPA$, and $\angle APB$, respectively. Then
	\begin{align}&\frac{PA+PB+PC}{2}\\  &\ge\frac{PA^2(PB+PC)^2+PB^2(PC+PA)^2+PC^2(PA+PB)^2}{(PB+PC)(PC+PA)(PA+PB)}\label{eq34a}\\
	&\ge l_a+l_b+l_c.\label{eq34b}
	\end{align}
	Equality at \eqref{eq34a} holds if and only if $PA=PB=PC$; equality at \eqref{eq34b} holds if and only if triangle $ABC$ is equilateral and $P$ is its center.
\end{theorem}
\begin{proof}Let $PA=p$, $PB=q$, and $PC=r$ then the first inequality comes from the following algebraic identity and inequality
	\begin{multline}\label{eq34c}(p+q+r)(p+q)(q+r)(r+p)-2p^2(q+r)^2-2q^2(r+p)^2-\\-2r^2(p+q)^2
	=pq(p-q)^2+qr(q-r)^2+rp(r-p)^2\ge 0.
	\end{multline}
	Equality at \eqref{eq34c} holds if and only if $p=q=r$ or $PA=PB=PC$. 
	
	For the second inequality, from Theorem \ref{thm4}, let
	\begin{equation}x=\left(\sqrt{\frac{PB}{PC}}+\sqrt{\frac{PC}{PB}}\right)^2,\ y=\left(\sqrt{\frac{PC}{PA}}+\sqrt{\frac{PA}{PC}}\right)^2,\end{equation}
	and
	\begin{equation}z=\left(\sqrt{\frac{PA}{PB}}+\sqrt{\frac{PB}{PA}}\right)^2.\end{equation}
	We obtain
	\begin{multline}\label{eq35}\frac{(PB+PC)^2}{PB\cdot PC}PA+\frac{(PC+PA)^2}{PC\cdot PA}PB+\frac{(PA+PB)^2}{PA\cdot PB}PC\ge\\ \left(\sqrt{\frac{PB}{PC}}+\sqrt{\frac{PC}{PB}}\right)\left(\sqrt{\frac{PC}{PA}}+\sqrt{\frac{PA}{PC}}\right)\left(\sqrt{\frac{PA}{PB}}+\sqrt{\frac{PB}{PA}}\right)\left(l_a+l_b+l_c\right),
	\end{multline}
	which is equivalent to
	\begin{equation}\label{eq36}\frac{PA^2(PB+PC)^2+PB^2(PC+PA)^2+PC^2(PA+PB)^2}{(PB+PC)(PC+PA)(PA+PB)}\ge l_a+l_b+l_c.
	\end{equation}
	Equality holds if and only if triangle $ABC$ is equilateral and $P$ is its center. This completes the proof of Theorem \ref{thm5}.
\end{proof}

Finally, we note that Theorem \ref{thm0} can be considered as generalization of "weighted Erd\H{o}s-Mordell inequality" which is given by  Dar and Gueron \cite{3} by setting $x=\frac{1}{u^2}$, $y=\frac{1}{v^2}$, and $z=\frac{1}{w^2}$ ($xyz=1$ since $uvw=1$), we obtain

\begin{theorem}[An equivalent form of weighted Erd\H{o}s-Mordell inequality in \cite{3}]\label{thm7}Let $u$, $v$, and $w$ be positive real numbers such that $uvw=1$. Let $P$ be an interior point of the triangle $ABC$. Denote by $d_a$, $d_b$, and $d_c$ the distances of $P$ from the sides $AB$, $BC$, and $CA$, respectively. Then
	\begin{equation}\frac{PA}{u^2}+\frac{PB}{v^2}+\frac{PC}{w^2}\ge 2(ud_a+vd_b+wd_c).
	\end{equation}
	Equality holds if and only if $u=v=w=1$, triangle $ABC$ is equilateral and $P$ is its center.
\end{theorem}
\section{Conclusion}
In this paper, we give a new weighted version of Erd\H{o}s-Mordell inequality, and also new weighted versions for Erd\H{o}s-Mordell type inequalities with proofs and some applications as the strengthened version of Erd\H{o}s-Mordell inequality and Erd\H{o}s-Mordell type inequalities.

\end{document}